\documentclass{amsart}
\usepackage{amsmath, amssymb, latexsym}
\title{Bounding Multiplicity by Shifts in the Taylor Resolution}
\author{Michael Goff}

\newtheorem{theorem}{Theorem}[section]
\newtheorem{proposition}[theorem]{Proposition}
\newtheorem{corollary}[theorem]{Corollary}
\newtheorem{lemma}[theorem]{Lemma}

\newtheorem{conjecture}[theorem]{Conjecture}

\newcommand{\K}{\Gamma}	
\newcommand{\F}{\mathcal{F}} 
\newcommand{\T}{\mathcal{T}} 
\newcommand{\field}{{\bf k}} 
\newcommand{\m}{{\mu}} 
\newcommand{\codim}{\mbox{\upshape codim}\,}

\newcommand{\LCM}{\mbox{\upshape lcm}\,}
\newcommand{\lk}{\mbox{\upshape lk}\,}
\newcommand{\GEN}{\mbox{\upshape GEN}\,}
\newcommand{\GCD}{\mbox{\upshape gcd}\,}
\newcommand{\len}{\mbox{\upshape len}\,}

\def\proof{\smallskip\noindent {\it Proof: \ }}
\def\proofof#1{\smallskip\noindent {\it Proof of #1: \ }}
\def\endproof{\hfill$\square$\medskip}

\begin{document}

\begin{abstract}
A weaker form of the multiplicity conjecture of Herzog, Huneke, and Srinivasan is proven for two classes of monomial ideals: quadratic monomial ideals and squarefree monomial ideals with sufficiently many variables relative to the Krull dimension.  It is also shown that tensor products, as well as Stanley-Reisner ideals of certain unions, satisfy the multiplicity conjecture if all the components do.  Conditions under which the bounds are achieved are also studied.
\end{abstract}

\date{August 15, 2007}

\maketitle
\section{Introduction}
In this paper we examine a relaxation of the multiplicity conjecture by using non-minimal free resolutions.

Throughout the paper we work with the polynomial ring
   $S=\field[x_1, \ldots, x_n]$
over an arbitrary field $\field$. If $I\subset S$ is a homogeneous ideal, 
then the ($\mathbb{Z}$-graded) \textit{Betti numbers} of $S/I$, 
$\beta_{i,j}=\beta_{i,j}(S/I)$,
   are the invariants that appear
in the minimal free resolution of $S/I$ as an $S$-module:
\[ 0 \rightarrow
\bigoplus_j S(-j)^{\beta_{l,j}} \rightarrow \ldots \rightarrow 
\bigoplus_j S(-j)^{\beta_{2,j}} \rightarrow \bigoplus_j 
S(-j)^{\beta_{1,j}} \rightarrow S \rightarrow S/I \rightarrow 0. \] 
Here $S(-j)$ denotes $S$ with grading shifted by $j$ and $l$ 
denotes the length of the resolution. In particular, $l \geq \codim(I)$.

Our main objects of study are the \textit{maximal and minimal shifts} 
in the resolution of $S/I$ defined by 
$M_i=M_i(S/I)=\max\{j : \beta_{i,j}\neq 0\}$ and 
$m_i=m_i(S/I)=\min \{j : \beta_{i,j}\neq 0\}$ for $i=1, \ldots, l$, respectively.
The following conjecture due to Herzog, Huneke, and 
Srinivasan \cite{HerzSr98} is known as the multiplicity conjecture.

\begin{conjecture}  \label{multiplicity-conj} 
Let $I\subset S$ be a homogeneous ideal of codimension $c$. 
Then the multiplicity of $S/I$, $e(S/I)$, satisfies the following upper bound:
$$ e(S/I) \leq ( \prod_{i=1}^{c} M_i )/c!.
$$
Moreover, if $S/I$ is Cohen-Macaulay, then also

$$ e(S/I) \geq ( \prod_{i=1}^{c} m_i )/c!.
$$
\end{conjecture}

The multiplicity conjecture was first motivated by a result of Huneke and Miller \cite{HuMill} which states that if $S/I$ is Cohen-Macaulay and has a pure resolution (that is, $m_i = M_i$ for $1 \leq i \leq c$), then $e(S/I) = \prod_{i=1}^c m_i / c!$.  Since then there has been much additional evidence, including many papers establishing the multiplicity conjecture for special classes of ideals.  Paper \cite{FrancSr} provides an excellent overview of the major results.  However, a general proof remains elusive.

We may instead use an arbitrary free resolution in the place of the minimal free resolution.  Let $I$ be a homogeneous ideal of $S$.  Let $\F'$ be the minimal free resolution of $S/I$ and let $\F$ be an arbitrary free resolution.  If $\beta_{ij}(\F)$ are the $\mathbb{Z}$-graded Betti numbers of $\F$, then $\beta_{ij}(\F) \geq \beta_{ij}(\F')$ by the minimality of $\F'$.  Let $M_i(\F)=\max\{j : \beta_{i,j}(\F)\neq 0\}$ and $m_i(\F)=\min \{j : \beta_{i,j}(\F)\neq 0\}$.  It follows that $M_i(\F) \geq M_i(\F')$ and $m_i(\F) \leq m_i(\F')$.  Hence we obtain a weaker from of the multiplicity conjecture.

\begin{conjecture}  \label{WeakMC} 
Let $I\subset S$ be a homogeneous ideal of codimension $c$.  Let $\F$ be an arbitrary free resolution of $S/I$.
Then $e(S/I)$ satisfies the following upper bound:
$$ e(S/I) \leq ( \prod_{i=1}^{c} M_i(\F) )/c!.
$$
Moreover, if $S/I$ is Cohen-Macaulay, then also

$$ e(S/I) \geq ( \prod_{i=1}^{c} m_i(\F) )/c!.
$$
\end{conjecture}

For an ideal $I$, Conjecture \ref{WeakMC} holds whenever Conjecture \ref{multiplicity-conj} holds.  We will refer to Conjecture \ref{WeakMC} as the $\F$-multiplicity conjecture.

In particular, we study Conjecture \ref{WeakMC} for the \textit{Taylor resolution} of a monomial ideal.  We refer to this case of Conjecture \ref{WeakMC} as the Taylor conjecture.  Partial cases of the Taylor conjecture were settled in
Corollary 4.3 and Theorem 5.3 of \cite{Herzog-Srinivasan2004}.  Suppose $I$ is a codimension $c$ ideal minimally generated by monomials $\GEN(I) := \{\m_1,\ldots,\m_r\}$.  The Taylor resolution is a cellular resolution, in the sense of \cite{MillSt}, supported on the labeled simplex with $r$ vertices, labeled $\m_j$, $1 \leq j \leq r$.  For more information on cellular resolutions, see Chapter 4 of \cite{MillSt}.  In particular, the $\mathbb{Z}$-graded Betti numbers of the Taylor resolution $\T$ are 
\begin{equation}
\label{TaylorBetti}
\beta_{ij}(\T) = |\{T \subset \GEN(I): |T| = i, \deg \LCM_{\m_k \in T} \m_k = j \}|, \quad 1 \leq i \leq c.
\end{equation}

We will denote the minimal and maximal shifts in the Taylor resolution by $\tilde{m}_i = \tilde{m}_i(S/I)$ and $\tilde{M}_i = \tilde{M}_i(S/I)$ respectively.  From (\ref{TaylorBetti}), we calculate 
\begin{equation}
\tilde{m}_i = \min \{|T|: T \subset \GEN(I), |T| = i, \deg \LCM_{\m_k \in T} \m_k = j \}
\end{equation}
and 
\begin{equation}
\label{TaylorM}
\tilde{M}_i = \max \{|T|: T \subset \GEN(I), |T| = i, \deg \LCM_{\m_k \in S} \m_k = j \}.
\end{equation}

The outline of the paper is as follows.  In Section \ref{Prelim} we review necessary background on the simplicial complexes and Stanley-Reisner ideals.  In Section \ref{TensorSection}, we prove that if $S/I$ and $S'/I'$ are two graded rings that satisfy the $\F$ and $\F'$-multiplicity conjectures respectively, then $S/I \otimes_\field S'/I'$ satisfies the $(\F \otimes \F')$-multiplicity conjecture.  In Section \ref{UnionSection}, we look at several results on when the $\F$ and $\F'$-multiplicity conjectures on Stanley-Reisner rings $S/I_\K$ and $S'/I_{\K'}$ imply the $\hat{\F}$-multiplicity conjecture on $\tilde{S}/I_{\K \cup \K'}$.  In Section \ref{LargeSection}, we establish the Taylor conjecture for squarefree monomial ideals of given Krull dimension $d$ and sufficiently many variables relative to $d$.  In Section \ref{QuadraticSection}, we prove the Taylor conjecture for quadratic monomial ideals and its upper bound part for monomial ideals for which all but one of the minimal generators has degree two.  We note that while this paper was in preparation, \cite{Kummini} appeared with an an alternate proof of the result on quadratic ideals.

\section{Preliminaries on Simplicial Complexes}
\label{Prelim}
In considering resolutions of monomial ideals, we often reduce to the case of squarefree monomial ideals via the method of \textit{polarization}.  We briefly recall this construction.  Let $S = \field[x_1,\ldots,x_n]$ as before, and let $I$ be a monomial ideal with $\GEN(I) = \{\m_1,\ldots,\m_r\}$ and $\m_j = x_1^{p_{1,j}}\ldots x_n^{p_{n,j}}$ for $1 \leq j \leq r$.  For $1 \leq i \leq n$, let $d_i$ be the maximum exponent of $x_i$ in $\GEN(I)$.  Let $$S' = \field[x_{1,1},\ldots,x_{1,d_1},\ldots ,x_{n,1},\ldots,x_{n,d_n}].$$  For $1 \leq j \leq r$, define $$\m_j' := \prod_{i=1}^n  x_{i,1}x_{i,2}\cdots x_{i, p_{i,j}} ,$$ and let $I' := (\m_1',\ldots,\m_r') \subset S'$.  We say that $I'$ is the \textit{polarization} of $I$.

The polarization $I'$ of a monomial ideal $I$ is a squarefree ideal, and $S'/I'$ has the same codimension, Betti numbers, and Taylor Betti numbers as $S/I$ \cite[pp. 44-45]{MillSt}.  Since the multiplicity of $S/I$ can be calculated from these invariants, $S/I$ and $S'/I'$ also have the same multiplicity.  Hence the following result holds.

\begin{proposition}
Let $I'$ be the polarization of a monomial ideal $I$.  Then $S/I$ satisfies the multiplicity / Taylor conjecture if and only if $S'/I'$ satisfies the multiplicity / Taylor conjecture.
\end{proposition}

The advantage of polarization is that every squarefree monomial ideal is the Stanley-Reisner ideal of a simplicial complex.  One can then use combinatorial and topological methods available for simplicial complexes to study the multiplicity and Taylor conjectures.

A \textit{simplicial complex} $\K$ is a collection of subsets, called \textit{faces}, of $[n]=\{1,2, \ldots, n\}$, such that $\K$ is closed under inclusion and for all $i \in [n]$, $\{i\} \in \K$.  We will also refer to $[n]$ as $V(\K)$, or the \textit{vertex set} of $\K$.  The \textit{dimension} of a face $F \in \K$ is $|F|-1$, while the \textit{dimension} of $\K$ is the largest dimension of a face of $\K$.  If $W \subseteq [n]$, let $\K[W]$ denote the \textit{induced subcomplex} on $W$.  The vertex set of $\K[W]$ is $W$, and the faces of $\K[W]$ are the faces of $\K$ that are contained in $W$.  As shorthand, we denote $\K[V(\K)-v]$ by $\K - v$.  If $F \in \K$, define the \textit{link} of $F$, denoted $\lk_{\K}(F)$, as the simplicial complex $\{G-F: F \subset G \in \K\}$.

If $\K$ is a simplicial complex on the vertex set 
$[n]:=\{1, 2, \ldots, n\}$, then its \textit{Stanley-Reisner ideal} 
(or the \textit{face ideal}), $I_\K$, 
is the ideal generated by the squarefree monomials 
corresponding to non-faces of $\K$, that is, $$ I_\K = (
       x_{i_1}\cdots x_{i_k} \, : \, \{i_1<\cdots <i_k\}\notin \K
              ), $$
and the \textit{Stanley-Reisner ring} (or the \textit{face ring}) 
of $\K$ is $S/I_\K$.  For more information on Stanley-Reisner rings, see \cite{BrHerz} and \cite{St96}.

We say that $\K$ is Cohen-Macaulay if $S/I_{\K}$ is Cohen-Macaulay.  We also say that $\K$ satisfies the $\F$-multiplicity conjecture when $S/I_\K$ satisfies it, and denote $m_i(S/I_\K,\F)$ and $M_i(S/I_\K,\F)$ by $m_i(\K,\F)$ and $M_i(\K,\F)$ respectively.

Various combinatorial and topological invariants of $\K$ 
are encoded in the algebraic invariants of $I_\K$ 
and vice versa \cite{BrHerz, St96}.
The Krull dimension of $S/I_\K$, $\dim  S/I_\K$, 
and the topological dimension of $\K$, $\dim \K$, are related by
$\dim  S/I_\K =\dim \K +1$ and so
 $$       
\codim(I_\K)=n-\dim\K-1.
$$

If $\F$ is the minimal free resolution, then Hochster's formula for the Betti numbers \cite[Theorem II.4.8]{St96} yields the following formulas for the minimal and maximal shifts of $\K$:

\begin{eqnarray}  \label{M-interpr}
M_i(\K)&=&\max\{|W| \, : \, W\subseteq[n] \mbox{ and }
              \tilde{H}_{|W|-i-1}(\K[W]; \field)\neq 0\}, \\ 
\label{m-interpr} m_i(\K)&=&\min\{|W| \, : \, W\subseteq[n] \mbox{ and }
              \tilde{H}_{|W|-i-1}(\K[W]; \field)\neq 0\}.
\end{eqnarray}
Here and in the rest of the paper, $\tilde H_i(\K; \field)$ denotes the $i$th reduced simplicial homology of $\K$ with coefficients in $\field$.  We also use $\tilde H_i(\K)$ when $\field$ is implicit.

The Hilbert series of $S/I_\K$ is 
determined by knowing the number of faces in each dimension.  
Specifically, let  $f_i = f_i(\K)$  be the number of $i$-dimensional faces.
 By convention, $f_{-1}=1$ with the empty set as the unique face of dimension minus one. 
 Then, 
$$
\sum^\infty_{i=0} \dim_\field (S/I_\K)_i \lambda^i = 
\frac{h_0 + h_1 \lambda +\dots + h_d \lambda^d}{(1-\lambda)^d},
$$ 
where, $(S/I_\K)_i$ is the $i$-th graded component of 
$S/I_\K$, $d=\dim\K+1=\dim  S/I_\K$, and
 \begin{equation} \label{h-vector}
    h_i = \sum^i_{j=0} (-1)^{i-j} \binom{d-j}{d-i} f_{j-1}.
\end{equation}
The multiplicity $e(S/I_\K)$ is $f_{d-1}(\K)$ which in turn is $h_0 + \dots + h_d.$

We also need the following definitions related to simplicial complexes.  Suppose $\K$ and $\K'$ are simplicial complexes.  We define the \textit{simplicial join} of $\K$ and $\K'$, $\K \star \K'$, as follows.  $V(\K \star \K') = V(\K) \coprod V(\K')$, and $$\K \star \K' = \{F \cup G: F \in \K, G \in \K'\}.$$ Hence the minimal non-faces of $\K \star \K'$ are precisely the minimal non-faces of $\K$ and $\K'$.  It follows that $(S \otimes_\field S') / I_{\K \star \K'} = (S/I_\K) \otimes_{\field} (S'/I_{\K'})$.

We say that $\K$ is a \textit{flag} simplicial complex if $I_{\K}$ is a quadratic ideal.  Equivalently, the minimal non-faces of $\K$ have two vertices.

Let $a = (a_1,\ldots,a_k)$ be a positive integer vector such that $\sum_{i=1}^{k}a_i = d$.  We say that a $(d-1)$-dimensional complex $\K$ is $a$-$balanced$ if there is a coloring of the vertices of $\K$ with colors $\{1,\ldots,k\}$ with the property that every face of $\K$ consists of at most $a_i$ vertices of color $i$ for all $1 \leq i \leq k$.  A \textit{balanced} complex, sometimes called a \textit{completely balanced} complex, is a $(1,\ldots,1)$-balanced complex.  Balanced and completely balanced complexes were introduced by Stanley in \cite{St79}.

\section {Tensor product of two resolutions}
\label{TensorSection}
In this section we prove that the multiplicity conjecture applies to the tensor products of two resolutions when it applies to the two resolutions individually.  Also, we characterize the circumstances under which the tensor product of two resolutions can be pure.  For the following theorem, let $(S/I_,\F)$ and $(S'/I',\F')$ be two (not necessarily Cohen-Macaulay) rings with free resolutions.  Let the codimensions of $I$ and $I'$ respectively be $c$ and $c'$.
\begin{theorem}
\label{TenRed}
If $S/I$ and $S'/I'$ satisfy the lower bound (resp. upper bound) inequalities of the $\F$- and $\F'$-multiplicity conjectures, then $(S/I) \otimes (S'/I')$ also satisfies the lower bound (resp. upper bound) inequality of the $\F \otimes \F'$-multiplicity conjecture.
\end{theorem}

If $\F$ and $\F'$ are the minimal free resolutions of $S/I$ and $S'/I'$ respectively, $\F \otimes \F'$ is the minimal free resolution of $S/I \otimes S'/I'$.  Similarly, if $I$ and $I'$ are monomial ideals and $\F$ and $\F'$ are the Taylor resolutions of $S/I$ and $S'/I'$, then $\F \otimes \F'$ is the Taylor resolution of $(S/I) \otimes (S'/I')$.  Hence we obtain the following corollary.

\begin{corollary}
\label{MinTay}
If $S/I$ and $S'/I'$ are rings that satisfy the multiplicity conjecture, then $S/I \otimes S'/I'$ satisfies the multiplicity conjecture.  If $I$ and $I'$ are monomial ideals such that $S/I$ and $S'/I'$ that satisfy the Taylor conjecture, then $S/I \otimes S'/I'$ satisfies the Taylor conjecture.
\end{corollary}

We remark that \cite[Theorem 1.1]{HerzogZheng} is an immediate consequence of Corollary \ref{MinTay}.

Our proof of Theorem \ref{TenRed} uses elementary operations on the sequences of minimal and maximal shifts associated with a resolution.  We will need a few lemmas to reduce to those operations.

Observe that if $(S/I,\F)$ and $(S'/I',\F')$ are graded rings with free resolutions, then $\F\otimes \F'$ is a free resolution of $(S/I)\otimes (S'/I')$, and we have $$\beta_{rs}(\F \otimes \F') = \sum_{i+j=r} \sum_{a+b=s}(\beta_{ia}(\F)+\beta_{jb}(\F')) \quad {\mbox{\upshape and}\,} \quad e(S/I \otimes S'/I') = e(S/I)e(S'/I').$$  Hence we obtain the following result.

\begin{lemma}
Let $m_i$, $m_i'$, and $\hat{m}_i$ be the minimal shifts of $\F$, $\F'$, and $\F \otimes \F'$ respectively, and let $M_i$, $M_i'$, and $\hat{M}_i$ be the maximal shifts.  Also set $m_0 = M_0 = 0$.  Then 
\begin{eqnarray*}
\hat{m}_r & = & \min \{ m_i+m'_j \ : \ i+j=r, \ i,j\geq 0\}  \quad \mbox{and} \\
\hat{M}_r & = & \max \{ M_i+M'_j \ : \ i+j=r, \ i,j\geq 0\}  
\end{eqnarray*}
\end{lemma}

Since the multiplicity conjecture only uses the first $c$ terms in a free resolution of a ring $S/I$ of codimension $c$, we also consider,

\begin{corollary}
Let $m_i$, $m_i'$, and $\hat{m}_i$ be the first $c$, $c'$, and $c+c'$ minimal shifts of $\F$, $\F'$, and $\F \otimes \F'$ respectively, and let $M_i$, $M_i'$, and $\hat{M}_i$ be the maximal shifts.  Also set $m_0 = M_0 = 0$.  Then 
\begin{eqnarray*}
\hat{m}_r & \leq & \min \{ m_i+m'_j \ : \ i+j=r, \ i,j\geq 0\}  \quad \mbox{and} \\
\hat{M}_r & \geq & \max \{ M_i+M'_j \ : \ i+j=r, \ i,j\geq 0\}  
\end{eqnarray*}
\end{corollary}

We will now define a \textit{lower join} operator $\star$ on sequences of positive real numbers in the following way.  Let $m = \{m_1,\ldots,m_k\}$ and $m' = \{m'_1,\ldots,m'_{k'}\}$.  Then $m \star m'$ is a sequence of positive real numbers of length $k+k'$ such that $$(m \star m')_r = \min\{m_i+m_j': i+j = r, i,j \geq 0\},$$ again with $m_0 = m_0' = 0$.  Similarly, define an \textit{upper join} operator $\bowtie$ so that if $M$ and $M'$ are sequences of positive real numbers of lengths $k$ and $k'$ respectively, and $M_0 = M'_0 = 0$, then $M \bowtie M'$ is a sequence of positive real numbers of length $k+k'$ with $$(M \bowtie M')_r = \max\{ M_i+M_j': i+j = r, i,j \geq 0\}.$$ Finally, define a function $F$ on sequences of positive real numbers by $$F\{m_1,\ldots,m_k\} := \frac{m_1\ldots m_k}{k!}.$$

Let $m$ and $m'$ be the first $c$ and $c'$ minimal shifts of $\F$ and $\F'$.  Since $m \star m'$ is the minimal shift sequence of $\F \otimes \F'$, and $S/I$ and $S'/I'$ satisfy the $\F$- and $\F'$-multiplicity lower bound conjectures, we can prove the $(\F \otimes \F')$-multiplicity lower bound conjecture on $(S/I) \otimes (S'/I')$ by proving that $F(m \star m') \leq F(m)F(m')$.  Similarly, we will prove the $(\F \otimes \F')$-multiplicity upper bound conjecture on $(S/I) \otimes (S'/I')$ by showing that if $M$ and $M'$ are the first $c$ and $c'$ maximal shifts of $\F$ and $\F'$, then $F(M \bowtie M') \geq F(M)F(M')$.

\begin{lemma}
\label{mSeq}
For sequences of positive real numbers $m$ and $m'$ of lengths $c$ and $c'$, $F(m \star m') \leq F(m)F(m')$.  Also, for sequences of positive real numbers $M$ and $M'$, $F(M \bowtie M') \geq F(M)F(M')$.
\end{lemma}
\proof We will prove the first statement.  The proof of the second statement is analogous and will be omitted.

Choose $a$ to be the minimum of all ${m_i}/{i}$ and ${m'_i}/{i}$, and let $$W = \{m_i: m_i = ai\} \cup \{m'_i: m'_i = ai\} = \{m_{s_1},\ldots,m_{s_j},m'_{s'_1},\ldots,m'_{s'_{j'}}\}.$$ Let $b$ be the second minimum value of ${m_i}/{i}$ and ${m'i}/{i}$ if such a $b$ exists.

Suppose $b$ exists.  Then $(m \star m')_i = ai$ for the following $j+j'$ distinct indices $i$: $s_1,\ldots,s_j,s_j+s'_1,\ldots,s_j+s'_{j'}$, and perhaps some others.  Hence, if we replace each $ai = m_i \in W$ by $bi$ and each $ai = m'_i \in W$ by $bi$, $F(m \star m')$ will increase by a factor of at least $({b}/{a})^{j+j'}$ and $F(m)F(m')$ will increase by a factor of exactly $({b}/{a})^{j+j'}$.  Hence we may make this substitution without loss of generality.

Repeat the above process until ${m_i}/{i} = {m'_j}/{j}$ for all $1 \leq i \leq c$ and $1 \leq j \leq c'$.  Then $m$ is of the form $(a,2a,\ldots,ca)$ and $m'$ is of the form $(a,2a,\ldots,c'a)$.  It follows that $m \star m' = (a,2a,\ldots,(c+c')a)$ and the desired inequality holds.
\endproof
\newline \newline
The proof of the above lemma not only implies Theorem \ref{TenRed}, but also gives very restrictive conditions under which equality can be attained.

\begin{theorem}
\label{TenEqLB}
Let $S/I$ and $S'/I'$ be Cohen-Macaulay rings that satisfy the $\F$- and $\F'$-multiplicity lower bound conjectures.  Then $(S/I) \otimes (S'/I')$ satisfies the $(\F \otimes \F')$-multiplicity lower bound conjecture with equality if and only if the following conditions hold: \newline
1) both $S/I$ and $S'/I'$ attain $\F$- and $\F'$-multiplicity lower bounds, and \newline
2) there exists a positive integer $a$ such that $m_i = ai$ for all $1 \leq i \leq c$ and $m'_i = ai$ for $1 \leq i \leq c'$.
\end{theorem}

\proof Since $S/I$, $S'/I'$, and $(S/I) \otimes (S'/I')$ are all Cohen-Macaulay, $m$, $m'$, and $m \star m'$ are the full minimal shift sequences of $S/I$, $S'/I'$, and $(S/I) \otimes (S'/I')$.

Assume without loss of generality $c \leq c'$.  The necessity of the first condition is clear from the inequality $F(m \star m') \leq F(m)F(m')$.  We will show that if the second condition fails, then equality fails in Theorem \ref{mSeq}.  Suppose that not all values of ${m_i}/{i}$ and ${m'_i}/{i}$ are the same.  Also suppose that we have increased the lowest values of ${m_i}/{i}$ and ${m'_i}/{i}$, as in the proof of Lemma \ref{mSeq}, to the point where ${m_i}/{i}$ and ${m'_i}/{i}$ only take on two values: namely $a$ and $b$ with $a<b$.

Let $W$ be constructed as in the proof of Theorem \ref{mSeq}.  If $(S/I) \otimes (S'/I')$ attains the $\F \otimes \F'$-multiplicity lower bound, then when we replace $ai = m_i \in W$ and $ai = m'_i \in W$ each by $bi$,  $F(m \star m')$ must increase by a factor of \textit{exactly} $({b}/{a})^{j+j'}$.  Hence, of the entries in $(m \star m')_i$, exactly $j+j'$ must be of the form $ai$ and the rest must be of the form $bi$.  We want to show that either all the ${m_i}/{i}$ and ${m'_i}/{i}$ are $a$ or they are all $b$, which is equivalent to $j+j'=0$ or $j+j' = c+c'$.  Suppose then, by way of contradiction, that $0 < j+j' < c+c'$.

Assume $(m \star m')_r = bi$ for some $1 \leq r \leq c+c'$.  Then for all $0 \leq i \leq c$ and $0 \leq i' \leq c'$ with $i+i' = r$, we have $m_i = bi$ and $m'_{i'} = bi'$.  Hence one of the following conditions hold: \newline
1) If $r \leq c$, then for some $1 \leq t \leq c$, all $m_i = bi$ for $i \leq t$, and for some $1 \leq t' \leq c'$, $m'_i = bi$ for $i \leq t'$. \newline
2) If $r \geq c'$, then for some $1 \leq t \leq c$ and $1 \leq t' \leq c'$, $m_i = bi$ for $i \geq t$,  and $m'_i = bi$ for $i \geq t'$. \newline
3) If $c < r < c'$, then $m_i = bi$ for all $1 \leq i \leq c$.

In Case 1, assume that $t$ and $t'$ are chosen maximally.  Then $a(t+t'+1) < \hat{m}_{t+t'+1} < b(t+t'+1)$, a contradiction.  In Case 2, assume $t$ and $t'$ are chosen minimally, and $t > 1$ and $t' > 1$.  Then $a(t+t'-1) < \hat{m}_{t+t'-1} < b(t+t'-1)$, a contradiction.  If $t=1$ or $t'=1$, without loss of generality suppose $t=1$, and then Case 3 applies.  In Case 3, let $i$ be the largest index so $m'_i = ai$; such an $i$ exists by hypothesis.  Then $\hat{m}_{i+c} = ai + bc$, a contradiction.  Hence we conclude that $F(m \star m') = F(m)F(m')$ only if $m = (a,2a,\ldots,ca)$ and $m' = (a,2a,\ldots,c'a)$.

Conversely, if both conditions are satisfied, then the minimal shift sequence for $\F \otimes \F'$ is $(a,2a,\ldots,(c+c')a)$, and the result follows.
\endproof

\begin{theorem}
\label{TenEqUB}
Let $S/I$ and $S'/I'$ be Cohen-Macaulay rings that satisfy the $\F$- and $\F'$-multiplicity upper bound conjectures.  Then $(S/I) \otimes (S'/I')$ satisfies the $(\F \otimes \F')$-multiplicity upper bound conjecture with equality if and only if the following conditions hold: \newline
1) both $S/I$ and $S'/I'$ attain the $\F$- and $\F'$-multiplicity upper bound, and \newline
2) there exists a positive integer $a$ such that $M_i = ai$ for $1 \leq i \leq c$ and $M'_i = ai$ for all $1 \leq i \leq c'$.
\end{theorem}

\proof The proof is similar to that of Theorem \ref{TenEqLB} and is omitted.
\endproof

In the case that $I$ and $I'$ are monomial ideals, our next theorem provides even stronger conditions under which equality is attained.

\begin{corollary}
\label{CrossPolytope}
Let $I$ and $I'$ be nonzero monomial ideals of $S$ and $S'$ respectively, and suppose $(S/I) \otimes (S'/I')$ has a pure resolution.  Then $I$, $I'$, and $(I \otimes 1') \oplus (1 \otimes I')$ are all generated by monomials in the same degree, say $a$.  Moreover, for every two minimal generators $\m_1$ and $\m_2$ of $I \otimes 1' \oplus 1 \otimes I'$, $\GCD(\m_1,\m_2) = 1$.
\end{corollary}
\begin{proof}
If $(S/I) \otimes (S'/I')$ has a pure resolution, then the minimal free resolution $\F$ of $(S/I) \otimes (S'/I')$ is pure.  By Theorem \ref{TenEqLB}, if $\F$ has length $k$, then $m(\F) = M(\F) = (a,2a,\ldots,ka)$.  Since neither $I$ nor $I'$ have codimension $0$, $k \geq 2$.  In particular, $m_1(\F) = M_1(\F) = a$, which implies that all generators of $(I \otimes 1') \oplus (1 \otimes I')$ have degree $a$.  Observe that $\GEN((I \otimes 1') \oplus (1 \otimes I')) = \GEN(I) \coprod \GEN(I')$, so the generators of $I$ and $I'$ are also all of degree $a$.

Since the minimal free resolution of $(S/I) \otimes (S'/I')$ is pure, $\beta_{2,r}((S/I) \otimes (S/I')) = 0$ whenever $r \neq 2a$.  Consider minimal generators $\m_1$ and $\m_2$ of $(I \otimes 1') \oplus (1 \otimes I')$ so that the LCM of $\m_1$ and $\m_2$ has degree $r$.  Then from the first syzygy of $\m_1$ and $\m_2$, $\beta_{2,r}((S/I) \otimes (S'/I')) > 0$, hence $r = 2a$ and $\GCD(\m_1,\m_2) = 1$.
\end{proof}

Suppose $\K$ and $\K'$ are simplicial complexes.  Then all of the above results apply to $\K \star \K'$.  By applying Corollary \ref{CrossPolytope} to Stanley-Reisner ideals, we obtain the following result.

\begin{corollary}
\label{CrossPolytopeSF}
If $\K \star \K'$ has a pure resolution, then one of the following conditions holds:\newline
1) $\K$ is a simplex and $\K'$ has a pure resolution, or vice versa, or \newline
2) each of $\K$, $\K'$, and $\K \star \K'$ is the join of a simplex and several copies of the boundary of the simplex on $a$ vertices.
\end{corollary}

We close this section with an interesting application of Theorem \ref{TenRed} to balanced simplicial complexes.  

\begin{theorem}
\label{UBBalanced}
Let $\K$ be an $(a_1,\ldots,a_k)$-balanced complex, where $a_i \leq 4$ for all $1 \leq i \leq k$.  Then $\K$ satisfies the Taylor upper bound conjecture.
\end{theorem}
\begin{proof}
For $1 \leq i \leq k$, let $V_i$ be the set of vertices of $\K$ colored $i$, and let $\K_i = \K[V_i]$.  Let $d-1$ be the dimension of $\K$.

$\K_i$ is a simplicial complex of dimension at most $3$.  It is shown in \cite{MGmin} that $\K_i$ satisfies the multiplicity upper bound conjecture, and hence $\K_i$ also satisfies the Taylor upper bound conjecture.  By Theorem \ref{TenRed}, $\K' := \star_{i=1}^b \K_i$ satisfies the Taylor upper bound conjecture.

Since $\K$ is a (non-induced) subcomplex of $\K'$, $f_{d-1}(\K) \leq f_{d-1}(\K')$.  Also, since $\K_i$ is an induced subcomplex of $\K$, $\GEN(I_{\K_i}) \subset \GEN(I_{\K})$ and hence $\GEN(I_{\K'}) \subset \GEN(I_{\K})$.  This implies $\tilde{M}(\K) \geq \tilde{M}(\K')$.  Hence $\K$ satisfies the Taylor upper bound conjecture.
\end{proof}

Equality in Theorem \ref{UBBalanced} is attained only when $\K = \K_1 \star \ldots \star  \K_k$ and the conditions of Corollary \ref{CrossPolytopeSF} apply.

\section{Unions of Simplicial Complexes}
\label{UnionSection}
In this section we consider some ways to express the multiplicity upper bound conjecture for a simplicial complex $\K$ in terms of the multiplicity upper bound conjecture for subcomplexes of $\K$.  This also provides our main inductive tool for the proof of Theorem \ref{FlagUB} below.

Throughout this section, we will use $U(\K,\F)$ or $U(\F)$ to refer to the upper bound on $e(S/I_\K) = f_{d-1}(\K)$ asserted by the $\F$-multiplicity conjecture.  If $I_\Gamma$ has codimension $c$, then $M(\F)$ is the sequence of the first $c$ maximal shifts of $S/I_\Gamma$.

The general principle used throughout this section is as follows.  Let $\K \cup \K'$ be a simplicial complex of dimension $d-1$.  If $\F$, $\F'$, and $\hat{\F}$ are free resolutions of $\K$, $\K'$, and $\K \cup \K'$ respectively, such that $U(\F) + U(\F') \leq U(\hat{F})$, and if $\K$ and $\K'$ satisfy the $\F$- and $\F'$-multiplicity upper bound conjectures, then $\K \cup \K'$ satisfies the $\hat{\F}$-multiplicity upper bound conjecture as well.  The reason is that $f_{d-1}(\K) + f_{d-1}(\K') \geq f_{d-1}(\K \cup \K')$.

More specifically, suppose $\K$ and $\K'$ are {\bf induced} subcomplexes of $\K \cup \K'$.  Also suppose $\F$, $\F'$, and $\hat{\F}$ are free resolutions of $\K$, $\K'$, and $\K \cup \K'$ respectively, so that when $M_i(\F)$ is defined, $M_i(\F) \leq M_i(\hat{\F})$ and when $M_i(\F')$ is defined, $M_i(\F') \leq M_i(\hat{F})$.  This condition is satisfied when $\F$, $\F'$, and $\hat{\F}$ are all minimal free resolutions or all Taylor resolutions.  If $\K \cup \K'$ has $\hat{n}$ vertices and dimension $d-1$, choose $t \geq 0$ so that $M_{\hat{n}-d}(\hat{\F}) = \hat{n}-d+t+1$.

One particularly important case is that of the minimal free resolution.  Say that a simplicial complex $\K$ is \textit{r-Leray} if for all $p \geq r$ and $W \subseteq V(\K)$, $\tilde{H}_p(\K[W]) = 0$.  Then $t$ is the maximum integer such that $\K \cup \K'$ is not $t$-Leray.  Equivalently, the Castelnuovo-Mumford regularity of $S/I_\K$ is $t+1$.

\begin{theorem}
\label{Union}
With the assumptions as above, if $f_0(\K \cap \K') \leq t+d-1$, then $\K \cup \K'$ satisfies the $\hat{\F}$-multiplicity upper bound conjecture.
\end{theorem}

In fact, we will prove the following stronger result.
\begin{proposition}
\label{UnionCor}
Assume $\K \not\subset \K'$ and $\K' \not\subset \K$.  Assume also that $\K$ has $n$ vertices and dimension $d-1$, and $\K'$ has $n'$ vertices and dimension $d'-1$.  If $d = d'$ and $f_0(\K \cap \K') \leq d+t-1$, then $$f_{d-1}(\K \cup \K') \leq U(\K \cup \K') + f_0(\K \cap \K') -(d+t-1),$$
while if $d' < d$, then $$f_{d-1}(\K \cup \K') \leq U(\K \cup \K') - n' + f_0(\K \cap \K'),$$
In particular, in Theorem \ref{Union}, $f_{d-1}(\K \cup \K') = U(\K \cup \K')$ only if $\K$ and $\K'$ have the same dimension $d-1$, $f_0(\K \cap \K')=d+t-1$, and $\K \cap \K'$ contains no faces of dimension $d-1$.
\end{proposition}

We make a few comments before the proof.  If $\hat{\F}$ is the minimal free resolution and $M_{\hat{n}-d}(\hat{\F}) = \hat{n}-d+1$, then $\K \cup \K'$ satisfies the multiplicity upper bound conjecture.  The reason is the well-known result that $\K \cup \K'$ is $1$-Leray and hence satisfies $f_{d-1}(\K) \leq \hat{n}-d+1 = M_1(\hat{\F})\ldots M_{\hat{n}-d}(\hat{\F})/(\hat{n}-d)!$.  (See \cite{Kalai84} for a much stronger result.)  Thus, for an arbitrary free resolution $\hat{\F}$, if $M_{\hat{n}-d}(\K \cup \K',\hat{\F}) = \hat{n}-d+1$, i.e. $t=0$, then $\K \cup \K'$ satisfies the $\hat{\F}$-multiplicity upper bound conjecture.  Thus we may assume without loss of generality that $t \geq 1$.  Then Theorem \ref{Union} implies that if $f_0(\K \cap \K') \leq d$, and $\K$ and $\K'$ satisfy the $\F$- and $\F'$-multiplicity upper bound conjectures, then $\K \cup \K'$ satisfies the $\hat{F}$-multiplicity upper bound conjecture.

In the case of a disjoint union, Proposition \ref{UnionCor} implies that $\K \cup \K'$ satisfies the $\hat{F}$-multiplicity upper bound conjecture with equality only if both $\K$ and $\K'$ are of dimension $0$.

For simplicity, we will refer to the quantity $(t+d-1)-f_0(\K \cap \K')$ as $z$.  By hypothesis, $z \geq 0$.

\proofof{Proposition~\ref{UnionCor}} Denote the length of a sequence of positive integers $M$ by $\len M$, and define $F(M) = \prod_{i=1}^{\len M}M_i/(\len M)!$.  In general, if $N$ is a sequence of length $r-1$,we can construct $N'$ from $N$ by appending a value $a \geq r+1$.  Then $F(N') \geq \frac{a}{r}F(N)$.  If $F(N) \geq r$, then $F(N') \geq F(N)+1$.

First we treat the case that $\K$ and $\K'$ have different dimensions, given respectively by $d-1$ and $d'-1$.  Without loss of generality, assume that $d' < d$, and that $\K$ and $\K'$ have respectively $n$ and $n'$ vertices.  Then $\K \cup \K'$ has dimension $d$.

Applying the above observation to $M(\hat{\F})$, and using the fact that $\len(M(\hat{\F})) = \len(M(\F)) + n'-f_0(\K \cap \K')$, we have $F(M(\hat{\F}) \geq F(M(\F)) + n' - f_0(\K \cap \K')$.  Also, $f_{d-1}(\K) = f_{d-1}(\K \cup \K')$, which proves Theorem \ref{Union} and Proposition \ref{UnionCor} in the case that $\K$ and $\K'$ have different dimensions.

Now consider the case that $\K$ and $\K'$ both have dimension $d-1$.  Then $\K \cup \K'$ has $n+n'-f_0(\K \cap \K')$ vertices, and $M(\hat{\F})$ has length $n+n'-d-f_0(\K \cap \K')$.  Suppose without loss of generality that $n' \leq n$.

Observe that $f_{d-1}(\K \cup \K') \leq f_{d-1}(\K) + f_{d-1}(\K')$, with equality exactly when $\K \cap \K'$ does not contain a face of dimension $d-1$.  Hence the theorem and proposition will follow if $$F(M(\F)) + F(M(\F')) \leq F(M(\hat{\F})) - z.$$  By hypothesis, $M(\hat{\F}) \geq M(\F)$ and $M(\hat{\F}) \geq M(\F')$ componentwise; hence we may replace $M(\hat{\F})$ by the componentwise minimal sequence $N$ such that $N \geq M(\F)$ and $N \geq M(\F')$ and prove 
\begin{equation}
\label{IntersectionEquation}
F(N) - F(M(\F)) - F(M(\F')) \geq z.
\end{equation}

$M(\F), M(\F') \leq N$, so we may replace $M_i(\F)$ with $N_i$ and $M_i(\F')$ with $N_i$ whenever both are defined since this operation decreases the left side of Equation (\ref{IntersectionEquation}).  Next, since $z \geq 0$, we may replace $M_i(\F)$, $M_i(\F')$, and $N_i$ by $i+1$ whenever all three are defined since this operation multiplies the left side of Equation (\ref{IntersectionEquation}) by a real number less than $1$.  By adding $F(M(\F'))$ to each side of Equation (\ref{IntersectionEquation}), we may similarly replace $M_i(\F)$ and $N_i$ with $i+1$ when the two are defined.  Finally, by adding $F(M(\F)) + F(M(\F'))$ to each side of Equation (\ref{IntersectionEquation}), we may similarly replace $N_i$ with $i+1$ when $i < n+n'-d-f_0(\K\cap \K')$.

By hypothesis, $$N_{n+n'-d-f_0(\K \cap \K')} = n+n'-d-f_0(\K \cap \K')+t+1.$$  Then, $F(M(\F)) = n-d+1$, $F(M(\F')) = n'-d+1$, and $F(N) = n+n'-d-f_0(\K \cap \K')+t+1$.  This yields $F(M(\F)) + F(M(\F')) \leq F(N)-z$ as desired.
\endproof

Next we prove another union related result that we will use in the proofs of Theorems \ref{LargeComplexUpper} and \ref{Quad}.  Its proof is a generalization of a calculation in \cite{NovSw} that is used to prove the multiplicity conjecture for matroid complexes.

\begin{lemma}
\label{HomRed}
Let $\K$ be a simplicial complex with dimension $d-1$, $n>d$ vertices, and free resolution $\F$ such that $M_{n-d}(\F)=n$.  For each $v \in V(\K)$, suppose $\K-v$ has free resolution $\F_v$ and $M_i(\F_v) \leq M_i(\F)$ for $1 \leq i \leq n-d-1$.  If $\K-v$ satisfies the $\F_v$-upper bound conjecture for all $v \in V(\K)$, then $\K$ satisfies the $\F$-upper bound conjecture.
\end{lemma}
\begin{proof}
Since every top-dimensional face of $\K$ contains $d$ vertices,
$$f_{d-1}(\K) = \frac{1}{n-d}\sum_{v \in V(\K)}f_{d-1}(\K_v) \leq \frac{1}{n-d}\sum_{v \in V(\K)}\frac{\prod_{i=1}^{n-d-1}M_i(\F_v)}{(n-d-1)!} \leq \frac{\prod_{i=1}^{n-d} M(\F)}{(n-d)!}.$$
\end{proof}

The condition that $M_i(\F_v) \leq M_i(\F)$ for $1 \leq i \leq n-d-1$ is satisfied if all resolutions are Taylor resolutions or if all resolutions are minimal free resolutions.

The method of reducing to unions of subcomplexes can be extended beyond induced subcomplexes, and to unions of more than two subcomplexes.  The proof of Theorem \ref{AlmostQuadratic} illustrates this principle.

\section{Large simplicial complexes}
\label{LargeSection}
The main theorem of this section is that if a simplicial complex $\K$ has sufficiently many vertices relative to its dimension, then $\K$ satisfies both bounds of the Taylor conjecture.  Furthermore, in this case $\K$ achieves neither of the Taylor bounds.  We will prove the upper bound and lower bound statements separately.

Suppose $\m = x_{i_1}x_{i_2}\ldots x_{i_r}$ is a minimal generator of $I_\K$, while $W \subseteq V(\K)$.  Say that $\m$ is \textit{supported} on $W$ if $\{i_j\}_{i=1}^r \subseteq W$.  If $Y =\{\m_1,\ldots,\m_t\}$ is a subset of minimal generators of $I_\K$, we say $Y$ is supported on $W$ if for each $1 \leq i \leq t$, $\m_i$ is supported on $W$.  Let $\tilde{L}(\K)$ and $\tilde{U}(\K)$ be the conjectured lower and upper Taylor bounds on $f_{d-1}(\K)$.
\begin{theorem} 
Let $\K$ be a simplicial complex of dimension $d-1$ and $n > 24d + 3$ vertices.  Then $\K$ satisfies the Taylor lower bound conjecture without equality.
\end{theorem}

\begin{proof}
Suppose there are $n-d$ distinct minimal generators of $I_{\K}$ that are supported on $n' < n$ vertices of $\K$.  Then $\tilde{m}(\K) \leq (n',n',\ldots,n')$ componentwise.  In that case, $$\tilde{L}(\K) \leq \frac{(n')^{n-d}}{(n-d)!} = \frac{(n-d)^{n-d}(\frac{n'}{n-d})^{n-d}}{(n-d)!} < (\frac{en'}{n-d})^{n-d}.$$  The last inequality follows from Stirling's approximation.  We then have $(\frac{en'}{n-d})^{n-d} < 1$ if $n' < (n-d)/{e}$, which occurs if $n' \leq n/3$ and $n > 11d$.  In this case, since $\K$ has at least one face of dimension $d-1$, $\K$ satisfies the Taylor lower bound conjecture without equality.

It thus suffices to prove the claim that if $n > 24d+3$, then there exists a set of $\lfloor n/3 \rfloor$ vertices that support $n-d$ monomials.  First, we will show that if $\Delta$ is an arbitrary simplicial complex of dimension at most $d-1$ and $n' > 4d$ vertices, then $I_{\Delta}$ has at least ${3(n')^2}/(8d)$ minimal generators.  If $\prod_{k=1}^t x_{i_k}$ is a minimal generator of $I_{\Delta}$ of degree at least $3$, we may without loss of generality replace $\prod_{k=1}^t x_{i_k}$ with $x_{i_1}x_{i_2}$ and delete all minimal generators of $I_{\Delta}$ that are multiples of $x_{i_1}x_{i_2}$.  Hence we may assume for the claim, without loss of generality, that $I_\Delta$ is quadratic, or that $\Delta$ is a flag complex.

Tur\'{a}n's theorem states that if $G$ is a graph that avoids cliques of size $d+1$, then $G$ has at most $(d-1)n^2/(2d)$ edges \cite{Bollobas}.  Since the graph of $\Delta$ avoids cliques of size $d+1$, Tur\'{a}n's theorem applies and $\Delta$ misses at least $$\frac{n'(n'-1)}{2} - \frac{(d-1)(n')^2}{2d} = \frac{n'(n'-d)}{2d} > \frac{3(n')^2}{8d}$$ edges.  Hence $I_\Delta$ has at least ${3(n')^2}/(8d)$ generators.

If $W \subset V(\K)$ and $|W| = n'$, then $I_{\K[W]}$ has at least $3(n')^2/(8d)$ minimal generators, each of which is a minimal generator of $I_\K$.  If $n > 24d + 3$ and $n' = \lfloor n/3 \rfloor$, then ${3(n')^2}/(8d) \geq n-d$, which proves the theorem.
\end{proof}

Observe that we did not assume that $\K$ is Cohen-Macaulay.  However, the Cohen-Macaulay assumption is necessary for complexes with few vertices.

With the hypothesis that $\K$ is completely balanced, we can tighten our bound on $n$.

\begin{theorem}
Let $\K$ be a Cohen-Macaulay completely balanced complex of dimension $d-1$ and $n \geq 3d$ vertices.  Then $\K$ satisfies the Taylor lower bound conjecture.
\end{theorem}
\begin{proof}
For $1 \leq i \leq d$, let $n_i$ be the number of vertices of color $i$, and suppose the colors are arranged so that $n_1 \geq n_2 \geq \ldots \geq n_d$.  Since $n \geq 3d$, $\sum_{i=1}^k n_i \geq 3k$ for $1 \leq k \leq d$.

Let $V_i$ be the set of vertices of color $i$.  Then $V_i$ supports ${n_i \choose 2}$ minimal generators of $I_\K$: namely all monomials of the form $x_sx_t$ for $s,t \in V_i$.  Since $n_1 \geq 3$, we conclude from the minimal generators supported on $V_1$ that $\tilde{m}_{{n_1 \choose 2}} \leq n_1$ and hence $\tilde{m}_{n_1} \leq n_1$.  Similarly 

\begin{equation}
\label{smallm}
\tilde{m}_{\sum_{i=1}^k {n_i \choose 2}} \leq \sum_{i=1}^k n_i \quad {\mbox{\upshape and}\,} \quad \tilde{m}_{\sum_{i=1}^k n_i} \leq \sum_{i=1}^k n_i.
\end{equation}

For $\sum_{i=1}^k n_i < r < \sum_{i=1}^{k+1} n_i$, we will construct a set of $r$ minimal generators supported on at most $r+1$ vertices.  First, by (\ref{smallm}), construct a set of $\sum_{i=1}^k n_i$ minimal generators supported on the $\sum_{i=1}^k n_i$ vertices of $V_1 \cup \ldots \cup V_k$.  With $q = \sum_{i=1}^k n_i$, label these minimal generators $\m_1, \ldots, \m_q$.  Then add the first $r - q$ minimal generators, ordered lexicographically, in $V_{k+1}$, which we will label $\m_{q+1},\ldots,\m_r$.  The support of $\{\m_{q+1},\ldots,\m_r\}$ consists of at most $r - q + 1$ vertices.  Hence $\tilde{m}_{r} \leq r+1$.  Using this and Equation (\ref{smallm}), we conclude that $\tilde{m}_r \leq r+1$ for all $1 \leq r \leq n-d$.  Hence $\tilde{L}(\K) \leq n-d+1$ and the Taylor lower bound conjecture holds since a $(d-1)$-dimensional Cohen-Macaulay complex has at least $n-d+1$ top-dimensional faces.
\end{proof}

We need the following lemma for the proof of the Taylor upper bound inequality.

\begin{lemma}
\label{IncM}
Let $I$ be a monomial ideal of $S$ with Taylor maximal shifts $M_1,\ldots,M_c$.  Suppose $S$ has $n$ indeterminants, of which $r$ appear in $\GEN(I)$.  If for $i < c$, $M_i < r$, then $M_{i+1} > M_i$.  If $M_i=r$, then $M_{i+1}=r$.
\end{lemma}
\begin{proof}
This follows immediately from Equation (\ref{TaylorM}).
\end{proof}

\begin{theorem}
\label{LargeComplexUpper}
Let $\K$ be a simplicial complex of dimension $d-1$ and $n \geq 9d+1$ vertices.  Then $\K$ satisfies the Taylor upper bound conjecture without equality.
\end{theorem}
\begin{proof}
If $\K$ is a cone with apex $v$, then $\K-v$ satisfies the conditions of the theorem.  Hence by induction on $d$ we may assume without loss of generality that $\K$ is not a cone.  Suppose $k$ is an integer so that $k \geq d$ and $n \geq 2k+d+1$.  Then $\tilde{M}(\K) \geq (2,4,\ldots,2k,k+d,k+d+1,\ldots)$.  We can see that $\tilde{M}_i \geq 2i$ for $1 \leq i \leq k$ inductively on $i$: if $\tilde{M}_{i-1} \geq 2k$, then $\tilde{M}_i \geq 2k$, while if $2(i-1) \leq \tilde{M}_{i-1} < 2k$, then consider $\mathcal{M} \subset \GEN(I_\K)$ with $|\mathcal{M}| = i-1$ such that $\mathcal{M}$ is supported on $\tilde{M}_{i-1}$ vertices.  There exist at least $d+1$ vertices not in the support of $\mathcal{M}$, which therefore support an additional minimal generator $\m$ of $I_\K$.  Hence $\tilde{M}_{i} \geq 2i$ by considering $\mathcal{M} \cup \{\m\}$.  The condition $\tilde{M}_i(\K) \geq i+d$ for $i > k$ follows from Lemma \ref{IncM} and the fact that $\K$ is not a cone.

If $\prod_{i=1}^k 2i \geq k!{k+d \choose d}$, then $\tilde{U}(\K) \geq {n \choose d}$, in which case the Taylor upper bound conjecture for $\K$ follows.  In turn, this inequality follows if $2^k \geq {k+d \choose d}$.  By Stirling's formula, the previous inequality follows if
$$
2^k > \frac{(k+d)^{k+d}e^k e^d}{e^{k+d}k^k d^d}.
$$
Taking the natural logarithm of both sides, the above follows if
$$
k \ln 2 + d \ln d + k \ln k \geq (k+d) \ln (k+d).
$$
Let $k = ad$.  Then, after simplification, the above equation is equivalent to
$$
a \ln 2 \geq (a+1) \ln \frac{a+1}{a} + \ln a.
$$

This is true if $a \geq 4$.  So $k \geq 4d$ and the Taylor upper bound conjecture holds when $n \geq 9d+1$.
\end{proof}

Our next result allows us to restrict to even smaller values of $n$ under suitable conditions when considering the Taylor upper bound conjecture.

\begin{lemma}
\label{n3d}
Let $\mathcal{C}$ be a class of simplicial complexes that is closed under induced subcomplexes.  Suppose every complex in $\mathcal{C}$ of dimension $d-1$ and fewer than $3d$ vertices satisfies the Taylor upper bound conjecture.  Then every complex in $\mathcal{C}$ of dimension $d-1$ satisfies the Taylor upper bound conjecture.
\end{lemma}

$\mathcal{C}$ can be the class of all simplicial complexes.  In Section \ref{QuadraticSection}, we use Lemma \ref{n3d} with $\mathcal{C}$ as the class of flag complexes.

\begin{proof}
Let $\K \in \mathcal{C}$, and suppose $\K$ has dimension $d-1$ and $n \geq 3d$ vertices.  If $\K$ is a cone, then without loss of generality we may remove the apex $v$ to obtain $\K'$ with $n-1$ vertices, dimension $d-2$, and $f_{d-2}(\K') = f_{d-1}(\K)$.  Since $n-1 \geq 3(d-1)$, the lemma applies to $\K'$.  Therefore, we will assume that $\K$ is not a cone.

We claim that $\tilde{M}_{n-d}(\K) = n$.  Assuming this claim, it follows by Lemma \ref{HomRed} and induction on $n$ that $\K$ satisfies the Taylor upper bound conjecture.  Since every set of $d+1$ vertices of $\K$ supports a minimal generator in $I_\K$, for some integer $r \geq 2d$ there exists $t$ disjoint minimal generators whose LCM has degree $r$.  Necessarily, $t \leq r-d$, and hence $\tilde{M}_t \geq r \geq t+d$.  Since $\K$ is not a cone, it follows from Lemma \ref{IncM} that $\tilde{M}_{n-d}(\K) = n$.  This proves the theorem.
\end{proof}

\section{Quadratic ideals}
\label{QuadraticSection}

Our main result of this section is the following.

\begin{theorem}
\label{Quad}
All quadratic monomial ideals satisfy the Taylor upper bound conjecture, and all Cohen-Macaulay quadratic monomial ideals satisfy the Taylor lower bound conjecture.
\end{theorem}

We will prove the lower bound and upper bound parts of Theorem \ref{Quad} separately.  Using polarization, we will assume $I = I_{\K}$ for some flag complex $\K$, and we will use $f_{d-1}(\K)$ as $e(S/I)$.  Then we will examine when equality on each bound is attained.

As before, we will use $\tilde{L}(\K)$ to denote the conjectured Taylor lower bound on $f_{d-1}(\K)$ and $\tilde{U}(\K)$ to denote the conjectured Taylor upper bound on $f_{d-1}(\K)$.

\begin{theorem}
\label{FlagLB}
Let $\K$ be a Cohen-Macaulay flag complex.  Then $\K$ satisfies the Taylor lower bound conjecture.
\end{theorem}
\begin{proof}
If $\K = \K_1 \star \K_2$, then $\K_1$ and $\K_2$ are both Cohen-Macaulay flag complexes.  Thus by Theorem \ref{TenRed}, we may assume without loss of generality that $\K$ is not the join of two complexes.

Let $G$ be the graph whose \textit{edge ideal} is $I_\K$, that is, the vertex set of $G$ is $[n]$, and $\{u,v\}$ is an edge in $G$ if and only if $x_u x_v \in I_\K$, or $\{u,v\}$ is not an edge in $\K$.  Since $\K$ is not the join of two simplicial complexes, $G$ is connected.  Therefore, there is an enumeration of the vertices of $G$, $(i_1,i_2,\ldots,i_n)$, with the following properties: for each $2 \leq t \leq n$, there exists $s_t \in [t-1]$ such that $i_{s_t}i_t$ is an edge in $G$.  Then for all $1 \leq t \leq n-1$, there exists $t$ minimal generators of $I_\K$ supported on at most $t+1$ vertices, namely $$\{x_{i_{s_2}}x_{i_2},x_{i_{s_3}}x_{i_3},\ldots,x_{i_{s_{t+1}}}x_{i_{t+1}}\}.$$ Hence $m_t \leq t+1$ for $1 \leq t \leq n-d$ and so $\tilde{L}(\K) \leq n-d+1$.

Since $\K$ is Cohen-Macaulay, $h_i(\K) \geq 0$ for $0 \leq i \leq d$.  Also, $h_{0}(\K) = 1$ and $h_1(\K) = n-d$.  It follows that $f_{d-1}(\K) \geq n-d+1$, proving the result.
\end{proof}

\begin{theorem}
\label{FlagUB}
Let $\K$ be a flag complex.  Then $\K$ satisfies the Taylor upper bound conjecture.
\end{theorem}
\begin{proof}
Let $\K$ have dimension $d-1$ and $n$ vertices.  Since all induced subcomplexes of $\K$ are also flag, then by Lemma \ref{n3d} we may assume without loss of generality that $n < 3d$.  Also, as in the proof of Theorem \ref{FlagLB}, we may assume without loss of generality that $\K$ is not the join of two nonempty simplicial complexes.

Let $G$ be the edge ideal of $I_\K$, as in the proof of Theorem \ref{FlagLB}.  Again, since $\K$ is not a join of two complexes, $G$ is connected.  We will consider two cases: first the case that $G$ has a vertex $v$ of degree $3$ or greater, and second the case that $G$ has no such vertex.

Assume $G$ has a vertex $v$ of degree at least $3$, with neighbors $u_1,u_2,u_3$.  Since $\dim \K = d-1$, all subsets of $d+1$ vertices of $\K$ support a minimal generator, and hence $\tilde{M}_i(\K) = 2i$ for $i \leq \lfloor \frac{n-d+1}{2}\rfloor$.  We see this by identifying $\m_1 \in \GEN(I_\K)$, removing the two vertices $x_1,y_1$ that support $\m_1$, identifying $\m_2 \in \GEN(I_{\K-\{x_1,y_1\}})$, and so on.  $\K$ cannot be written as the join of two complexes, so in particular $\K$ is not a cone.  Hence by Lemma \ref{IncM}, $M_i \geq i + \lfloor \frac{n-d+1}{2}\rfloor$ for $i \geq \lfloor \frac{n-d+1}{2}\rfloor$, which yields $$M_{n-d} \geq n-d+\lfloor \frac{n-d+1}{2}\rfloor \geq \frac{3}{2}(n-d).$$  Similarly, $M_t \geq \frac{3}{2}t$ for all $1 \leq t \leq n-d$.

Since $vu_1, vu_2, vu_3$ are not edges in $\K$, $\K$ is a union of induced subcomplexes $(\K-v)$ and $(\K-\{u_1,u_2,u_3\})$.  Then, as in Section \ref{UnionSection}, we can inductively reduce the Taylor upper bound conjecture on $\K$ to the Taylor upper bound conjecture on $\K_1 = \K-v$ and $\K_2 = \K-\{u_1,u_2,u_3\}$ if we can show that $\tilde{U}(\K) \geq \tilde{U}(\K_1)+\tilde{U}(\K_2)$.

The Taylor maximal shift sequence is nonincreasing under induced subcomplexes.  Hence $$\tilde{U}(\K_1) \leq \frac{n-d}{\tilde{M}_{n-d}}\tilde{U}(\K) \leq \frac{2}{3}\tilde{U}(\K).$$  By a similar calculation $\tilde{U}(\K_2) \leq \frac{8}{27}\tilde{U}(\K)$.  It follows that $\tilde{U}(\K) \geq \tilde{U}(\K_1)+\tilde{U}(\K_2)$ as desired.  This completes the case that $G$ has a vertex of degree $3$ or greater.

Now we consider the case that $G$ does not have a vertex of degree three or greater.  Since $G$ is connected, $G$ is either a path or a cycle.  Without loss of generality, suppose $12,23,\ldots,(n-1)n$ are edges in $G$.  Then $x_1x_2,x_3x_4,\ldots,x_{2\lfloor \frac{n}{2} \rfloor-1}x_{2\lfloor \frac{n}{2} \rfloor} \in I_\K$, so $\tilde{M}_i = 2i$ for $i \leq \lfloor \frac{n}{2} \rfloor$.  Again by Lemma \ref{HomRed}, we may assume without loss of generality $\tilde{M}_{n-d} < n$.  Then by Lemma \ref{IncM}, $\tilde{M} = (2,4,\ldots,2(n-d))$ or $\tilde{M} = (2,4,\ldots,2(n-d-1),2(n-d)-1)$.  If $n=2$, then $\K$ is a pair of isolated vertices and satisfies the Taylor upper bound conjecture.  Assume $n \geq 3$, so that $G$ contains a vertex $v$ with two neighbors: $u_1$ and $u_2$.  Let $\K_1 = \K-v$ and $\K_2 = \K-\{u_1,u_2\}$.  Then $\K = \K_1 \cup \K_2$.  It is easy to verify that $U(\K) \geq U(\K_1)+U(\K_2)$ by calculations similar to those above, so that we may inductively reduce the Taylor upper bound conjecture on $\K$ to the Taylor upper bound conjecture on $\K_1$ and $\K_2$.

We have shown that we may inductively reduce all flag complexes to simplices either by $\tilde{M}_{n-d}=n$ and applying Lemma \ref{HomRed}, by expressing $\K$ as the join of two flag complexes, or by expressing $\K$ as the union of two flag complexes.  Hence all flag complexes satisfy the Taylor upper bound conjecture.
\end{proof}

We now turn our attention to the question of when these bounds are attained, starting with the lower bound.  We will focus on the case $I = I_\K$ for a Cohen-Macaulay flag complex $\K$.  By Corollary \ref{CrossPolytope}, if $\K = \K_1 \star \K_2$ and neither $\K_1$ nor $\K_2$ are simplices, then $\K$ attains the lower bound only if $\K$ is the join of the boundary of a cross polytope and a simplex, and otherwise we have $\tilde{L}(\K) \leq n-d+1$.  If $\tilde{L}(\K) \leq n-d+1$, then since $\K$ is Cohen-Macaulay, $\K$ attains the lower bound if and only if $f_{d-1}(\K) = n-d+1$ and $\tilde{L}(\K) = n-d+1$, which is equivalent to $\tilde{m}_i = i+1$ for all $1 \leq i \leq n-d$.

If $\K$ is Cohen-Macaulay and $f_{d-1}(\K) = n-d+1$, we say that $\K$ is a \textit{generalized tree}.  Equivalently, there is an enumeration of the facets of $\K$, $F_1, F_2, \ldots, F_{n-d+1}$, such that for $2 \leq i \leq n-d+1$, $F_i \cap (F_1 \cup \ldots \cup F_{i-1})$ is a face of dimension $d-2$.  In the case $d=2$, a generalized tree is a tree in the usual graph theoretic sense.

\begin{proposition}
Suppose $\K$ is a Cohen-Macaulay flag simplicial complex and $\K$ attains the Taylor lower bound.  Then, up to isomorphism, $\K$ is the join of a simplex and one of the following: \newline
1) two isolated vertices, \newline
2) a path of length four, \newline
3) the two-dimensional complex on six vertices with facets $\{123,234,345,456\}$, \newline
4) the boundary of a cross polytope.
\end{proposition}
\begin{proof}
If $\K$ is a cone, we can without loss of generality remove the apex vertex from $\K$.  Hence we will assume $\K$ is not a cone.  If $\K = \K_1 \star \K_2$ and neither $\K_1$ nor $\K_2$ are simplices, then by Theorem \ref{CrossPolytope}, condition 4 applies.  Henceforth we will assume this is not the case.  Then $\K$ must be a generalized tree and $\tilde{m}_i(\K) = i+1$ for $1 \leq i \leq n-d$.

If $\K$ has three mutually disconnected vertices, then $\tilde{m}_3(\K) = 3$ and $\K$ misses the Taylor lower bound.  Thus if $\K$ has dimension at most one, it is easy to see $\K$ must satisfy one of the conditions above.

Suppose $d \geq 3$.  If $\K$ attains the Taylor lower bound, then $\K$ does not have three mutually disconnected vertices and thus has exactly two vertices of degree $d-1$: $u$ and $v$.  Since $\K$ is not a cone, $\K$ contains at least $2d$ vertices.  $\K-\{u,v\}$ contains at least $2d-2$ vertices and is also a generalized tree.  There exists $u' \in V(\K)-\{u,v\}$ such that $uu'$ is an edge in $\K$ and $u'$ has degree $d-1$ in $\K-\{u,v\}$.  Similarly, there exists $v' \in V(\K)-\{u,v\}$ such that $vv'$ is an edge in $\K$ and $v'$ has degree $d-1$ in $\K-\{u,v\}$.  It follows that $uv, uv', u'v, u'v'$ are not edges in $\K$, and if $n-d \geq 4$, $\tilde{m}_4(\K) = 4$.  Hence $\K$ cannot attain the Taylor lower bound if $n-d \geq 4$.  It follows that $n-d = 3$ and $d=3$ by $n \geq 2d$.  Condition 3 applies in this case.
\end{proof}

If $\K$ attains the Taylor upper bound with equality and cannot be written as a join of two complexes, then $G$ in the proof of Theorem \ref{FlagUB} cannot have a vertex of degree $3$.  Hence $G$ is either a path of a cycle.  It is easy to see from the proof that the inequality $\tilde{U}(\K) \geq \tilde{U}(\K_1) + \tilde{U}(\K_2)$ is an equality only if $n=2$ or $n=3$.

\begin{proposition}
Let $\K$ be a flag complex.  Then $\K$ attains the Taylor upper bound if and only if $\K$ is the join of a simplex with one of the following: \newline
1) the boundary of a cross-polytope, \newline
2) two isolated vertices, \newline
3) three isolated vertices.
\end{proposition}

We conclude with an extension of the Taylor upper bound conjecture to ideals that are ``almost'' quadratic.

\begin{theorem}
\label{AlmostQuadratic}
Let $I$ be a monomial ideal minimally generated by monomials $\m_1, \m_2, \ldots, \m_r$.  Suppose for $2 \leq i \leq r$, $\m_i$ has degree $2$.  Then $S/I$ satisfies the Taylor upper bound conjecture.
\end{theorem}
\begin{proof}
We will prove the result by induction on the degree of $\m_1$.  If $\m_1$ has degree $2$, then $I$ is a quadratic ideal, and $S/I$ satisfies the Taylor upper bound conjecture by Theorem \ref{Quad}.

By polarization, we may assume without loss of generality that $I$ is a squarefree monomial ideal.  Also without loss of generality, $\m_1 = x_1x_2\ldots x_t$.  Let $I = I_\K$ for a simplicial complex $\K$, and suppose $\K$ has dimension $d-1$ and $n$ vertices.

Suppose $t \geq 3$.  Since $\m_1$ is a minimal generator of $I_\K$, $[t]-\{i\}$ is a face in $\K$ for $1 \leq i \leq t$.  Let $\K_i = \K-\{F \in \K: [t]-\{i\} \subseteq F\}$.  Every face in $\K$ contains at most one face of the form $[t]-\{i\}$, so $$\sum_{i=1}^t f_{d-1}(\K_i) \geq (t-1)f_{d-1}(\K).$$

The minimal generators of $I_{\K_i}$ are the same as the minimal generators of $I_\K$, except $\m_1$ is replaced by $\m_1/x_i$.  Hence $\tilde{M}_j(\K_i) \leq \tilde{M}_j(\K)$ for $1 \leq j \leq n-d$.  Also, $M_1(\K) = t$, whereas $M_1(\K_i) = t-1$.  It follows that $\tilde{U}(\K_i) \leq \frac{t-1}{t}\tilde{U}(\K)$.  For $1 \leq i \leq t$, $\K_i$ satisfies the Taylor upper bound conjecture by the inductive hypothesis.  Hence $\K$ satisfies the Taylor upper bound conjecture as well.
\end{proof}

\section*{Acknowledgements}
The author wishes to thank Isabella Novik for much helpful discussion and guidance, and also Manoj Kummini for many helpful comments on the draft version of this paper.

\end{document}